\documentclass[10pt,a4paper]{article}

\usepackage[utf8]{inputenc} % not necessary since 2018
\usepackage[T1]{fontenc} % to manage PDF

\usepackage[english]{babel}
\usepackage{amsmath}
\usepackage{amsfonts}
\usepackage{amssymb}
\usepackage{dsfont} % to make indicator function
\usepackage{amsthm} % to manage theorems

\usepackage{csquotes} % to manage quotations
\usepackage{graphicx} % more options for \includegraphics (colors,...)
\usepackage{xcolor} % to manage colors
\usepackage{enumitem} % to manage the size between the lines [itemsep=20pt]
\usepackage{hyperref} % to have hyperlinks in the document
\usepackage{verbatim} % to use \begin{comment}- \end{comment}

\usepackage[capitalise]{cleveref}
\usepackage[backend=bibtex,citestyle=alphabetic, bibstyle=alphabetic,isbn=false,maxcitenames=10, maxnames=10]{biblatex}
\DeclareSourcemap{
  \maps[datatype=bibtex]{
    \map{
      \step[fieldsource=doi,final]
      \step[fieldset=url,null]
      \step[fieldset=urldate,null]
    }
  }
}

\usepackage{ stmaryrd }

\newcommand*{\dd}{\mathop{}\!\mathrm{d}}

\newcommand\R{\mathbb{R}}

\newcommand\N{\mathbb{N}}

\newcommand\fctset[3]{{#1}: {#2} \rightarrow {#3}}

\newcommand\soe{\geqslant}
\newcommand\ioe{\leqslant}

\newcommand{\fsL}{\textnormal{L}} % Lp spaces
\newcommand{\fsW}{\textnormal{W}} % Sobolev spaces
\newcommand{\Lip}{\mathrm{Lip}} % Lipschitz
\newcommand{\init}{\mathrm{in}} % initial data

\newcommand\Ccal{\ensuremath{\mathcal C}}

\renewcommand\a{\alpha}

\renewcommand\phi{\varphi}
\newcommand\eps{\varepsilon}

\newcommand\indic[1]{\mathds{1}_{#1}}

\newtheorem{theo}{Theorem}

\newtheorem{lemm}{Lemma}

\newtheorem{prop}{Proposition}

\newtheorem{rema}{Remark}

\newtheorem*{theo*}{Theorem}
\newtheorem*{coro*}{Corollary}
\newtheorem*{lemm*}{Lemma}
\newtheorem*{conj*}{Conjecture}
\newtheorem*{prop*}{Proposition}
\newtheorem*{defi*}{Definition}
\newtheorem*{rema*}{Remark}

\newcommand{\B}{\textnormal{B}}
\DeclareMathOperator{\oscillation}{osc}

\author{Hector Bouton
  \thanks{Email: \href{bouton@imj-prg.fr}
    {\texttt{bouton@imj-prg.fr}}\\
  Université Paris Cité and Sorbonne Université, CNRS, IMJ-PRG, F-75013 Paris, France}
  \and
  Laurent Desvillettes
  \thanks{Email: \href{mailto:desvillettes@imj-prg.fr}
    {\texttt{desvillettes@imj-prg.fr}}
    \\
    Université Paris Cité and Sorbonne Université, CNRS and IUF, IMJ-PRG, F-75013 Paris, France.}
   \and
  Helge Dietert
   \thanks{Email: \href{mailto:helge.dietert@imj-prg.fr}
    {\texttt{helge.dietert@imj-prg.fr}} \\
    Université Paris Cité and Sorbonne Université, CNRS, IMJ-PRG, F-75013 Paris, France.}
}
\title{Hölder regularity of parabolic equations with Dirichlet boundary conditions  and application to reaction-diffusion and reaction-cross-diffusion systems}

\begin{document}
\maketitle
\begin{abstract}
  In this work, we adapt our recent article \cite{boutondesvillettes2025} to the
  setting of Dirichlet boundary conditions.  A key part is the study of the
  parabolic equation \(a\partial_t w - \Delta w = f\) with a rough coefficient \(a\),
  homogeneous Dirichlet boundary conditions, and the special assumption
  \(\partial_tw \ge 0\).  We then apply it to prove existence of global strong solutions
  to the triangular Shigesada-Kawasaki-Teramoto (SKT) cross-diffusion system
  with Lotka-Volterra reaction terms in three dimensions and Dirichlet boundary
  conditions, and to obtain estimates for solutions to reaction-diffusion
  systems modeling reversible chemistry (still when Dirichlet boundary
  conditions are considered).
\end{abstract}

\section{Introduction}

In this paper, we cover parabolic equations of the form
\begin{equation}
  \label{eq:rough-scalar-pde}
  a(t,x) \partial_t w(t,x) - \Delta w(t,x) = f(t,x)
  \qquad \text{and}  \quad \partial_t w \ge 0,
\end{equation}
for the unknown \(w = w(t,x)\) on $[0, T]\times\Omega$, where $\Omega$ is a bounded $C^2$
domain of $\R^d$.  Here $a:=a(t,x)$ is a non-constant coefficient, and
$f:=f(t,x)$ is a source term.

In \cite{boutondesvillettes2025}, we proved under the ellipticity condition on
$a$, that is
\begin{equation}
  \label{eq:assumption-a}
  0 < a_0 \le a \le c_0 a_0 < \infty,
\end{equation}
for some constants \(a_0\) and \(c_0 \ge 1\), and with homogeneous Neumann boundary
data
\begin{equation} \label{neu}
  \vec n \cdot \nabla_x w = 0 \qquad {\hbox{ on }} \qquad  (0,T] \times \partial\Omega,
\end{equation}
the Hölder regularity of \(w\) provided that \(f \in \fsL^{p}((0,T]; \fsL^q(\Omega))\)
with $\gamma := 2 - \frac2p - \frac{d}q > 0$.

The obtained estimate can be seen as a special case of the results proven in
1981 by Krylov and Safonov \cite{krylov1981}.  The extra assumption
$\partial_t w \soe 0$ enables to write down a direct proof in which one systematically uses comparison  with solutions to the
standard heat equation.  It also enables  a complete treatment of boundary conditions
with explicit constants, and, moreover, it requires only a control of \(f\) in the critical
Lebesgue space.
\medskip

In this work, we adopt this strategy to the case of Dirichlet boundary
conditions. More precisely, we provide the

\begin{theo}\label{thm:thm-dirichlet-bc}
  We consider a bounded, $\Ccal^2$ domain \(\Omega \subset \R^d\), where $d \in \N - \{0\}$.  Set \(T>0\), and
  \(p, q \in [1, \infty]\) such that $\gamma := 2 - \frac2p - \frac{d}q > 0$.  Then there
  exists a constant \(\alpha >0\) only depending on \(\gamma, d,c_0\), and a constant
  \(C_*\) only depending on \(p,q,d, \Omega, T, a_0,c_0\) such that for any
  $\Ccal^1(\overline\Omega)$ initial data $w_\init$ satisfying also the homogeneous Dirichlet
  boundary condition $w_\init(x) = 0$ for $x\in\partial\Omega$, forcing data
  \(f \in \fsL^{p}((0,T]; \fsL^q(\Omega))\) and a coefficient \(a := a(t,x)\)
  satisfying the bound \eqref{eq:assumption-a}, a solution \(w \ge 0\) of
  \eqref{eq:rough-scalar-pde} over \((0,T] \times \Omega\) with homogeneous Dirichlet
  boundary data ($w = 0$ on $ (0,T] \times \partial\Omega$) lies in
  \(\Ccal^{0,\alpha}([0,T] \times {\overline{\Omega}} ) \).  Moreover, the following estimate
  holds:
  \begin{equation*}
    \| w \|_{\Ccal^{0,\alpha}([0,T] \times {\overline{\Omega}} ) }
    \le C_*
    \left(
      \| f \|_{\fsL^p((0,T]; \fsL^{q}(\Omega))}
      + \| w_{\init} \|_{\Ccal^1(\overline\Omega)}
    \right).
  \end{equation*}
\end{theo}
\begin{rema}
  Similarly as in \cite{boutondesvillettes2025}, we can improve the bound to the more precise estimate
  \begin{multline*}
    \| w \|_{\Ccal^{0,\alpha}([0,T] \times {\overline{\Omega}} ) }
    \le C_*
    \left(
      \| f_+ \|_{\fsL^p((0,T]; \fsL^{q}(\Omega))}
      + \| w_{\init} \|_{\Ccal^1(\overline\Omega)}
    \right)^{1-\alpha/\gamma}\\
    \left(
      \| f\|_{\fsL^p((0,T]; \fsL^{q}(\Omega))}
      + \| w_{\init} \|_{\Ccal^1(\overline\Omega)}
    \right)^{\alpha/\gamma}.
  \end{multline*}
  We can also explicitly quantify $\a$ and $C_*$.
\end{rema}

\begin{rema}
  The assumption $w_\init\in\Ccal^1(\overline\Omega)$ is not optimal but leads to simplified proofs, it can be replaced by $w_\init\in\Ccal^{0,\beta}(\overline\Omega)$ for any $0 < \beta < 1$.
\end{rema}

The proof in \cite{boutondesvillettes2025} is based on a reduction of
oscillations in the sets $\B(x, r)\cap\Omega$, for $r>0$, $x\in\Omega$. In the setting of
homogeneous Dirichlet boundary conditions, such a reduction of oscillations only
holds when $B(x, r)\subset \Omega$. However, we are able to recover the Hölder regularity
of the solution in the whole domain, since the solution converges to $0$ near
the boundary.\medskip

As an application, we consider a classical model coming from reversible
chemistry. Defining for $i = 1, \dots, 4$ the concentrations
$u_i := u_i(t, x) \soe 0$ of species $A_i$ diffusing with rate $d_i$, and taking
into account the reversible reaction $A_1 + A_3 \rightleftharpoons A_2 + A_4$, we end up
(assuming that the rates of reaction is $1$) with the system
\begin{equation}\label{eq:chemistry}
  \partial_t u_i - d_i\Delta u_i = (-1)^i(u_1\,u_3 - u_2\,u_4).
\end{equation}
In \cite{boutondesvillettes2025}, using the H\"older estimate for the solution
to \eqref{eq:rough-scalar-pde}, \eqref{neu}, we provided a new simple proof of existence of
strong global solutions to this system when it is complemented with homogeneous
Neumann boundary conditions and the dimension is $d \le 4$.

Considering now the Dirichlet boundary condition
\begin{equation}\label{eq:boundary-chemistry}
  u_i(t,x) = 0\quad \text{ for } (t,x) \in \R_+ \times \partial\Omega,
\end{equation}
we obtain the following result:

\begin{prop} \label{prord} Let $d \le 4$ and $\Omega \subset \R^d$ be a $\Ccal^2$ bounded
  domain.  Let $d_i >0$ be constant diffusion rates for $i = 1, \dots, 4$, and
  suppose that the initial data $u_i^\init \ge 0$ lie in $\fsW^{2,\infty}(\Omega)$ and are
  compatible with the homogeneous Dirichlet boundary condition
  \eqref{eq:boundary-chemistry}.

  Then there exists a strong (that is, all terms appearing in the equation are
  defined a.e.) solution to system
  \eqref{eq:chemistry}--\eqref{eq:boundary-chemistry} with initial data
  $u_i^\init$.
\end{prop}

Note that existence of strong solutions to \eqref{eq:chemistry} with Neumann boundary
conditions (or in the whole space) in all dimension $d$ was obtained by different
methods in \cite{caputogoudon2019, fellner2020, souplet2018}, and that the proofs
can be adapted rather simply to the Dirichlet case.

\begin{rema}
  The techniques that we use in the proof of \cref{prord} apply also verbatim
  to the more general framework of mass-dissipative reaction-diffusion equations
  with quadratic non-linearity, we refer to \cite{boutondesvillettes2025} and
  references therein for the definition and treatment of such systems (with homogeneous Neumann boundary conditions).
\end{rema}

We also study the triangular Shigesada-Kawasaki-Teramoto (SKT) system,
describing two interacting biological species with number density
$u:=u(t,x), v:=v(t,x) \soe 0$. The competition is supposed to be of
Lotka-Volterra type, and the diffusion rate of the species $u$ increases when
the concentration of the species $v$ increases. More precisely, the model writes
\begin{equation}\label{eq:SKT}
  \left\{
  \begin{aligned}
    \partial_t u - \Delta[ (d_1 +\sigma v) u ] = f_u (u,v), \\
    \partial_t v - d_2\,\Delta v = f_v (u,v), \\
  \end{aligned}
  \right.
\end{equation}
where 
\begin{align*}
  f_u(u,v) &= u\,( r_u - d_{11}\, u - d_{12}\, v), \\
  f_v(u,v) &= v\,(r_v - d_{21}\, u - d_{22}\, v),
\end{align*}
and $d_1,d_2 ,\sigma, d_{ij} >0$ (for $i,j=1,2$), $r_u,r_v\ge 0$,  are constants.
\medskip

In \cite{boutondesvillettes2025}, still using the H\"older estimate for the
solution to \eqref{eq:rough-scalar-pde} with homogeneous Neumann boundary
condition \eqref{neu}, we proved existence of strong global solutions to this
system in dimension $d \le 4$ when it is complemented with homogeneous Neumann
boundary conditions.

We also can use Theorem~\ref{thm:thm-dirichlet-bc} to prove the existence of
strong solutions to system \eqref{eq:SKT} with homogeneous Dirichlet boundary
condition
\begin{equation}\label{eq:bc-SKT}
  u(t, x) = 0 \quad \text{ and }\quad v(t, x) = 0  \quad  \text{ for } (t,x) \in [0, T] \times \partial\Omega .
\end{equation}

We get the 

\begin{prop}\label{thm:SKT-dir}
  Let $d\ioe 4$, $\Omega \subset \R^d$ be a $\Ccal^2$ bounded domain, and consider $d_1,d_2,\sigma >0$,
  $r_u,r_v\ge 0$, $d_{ij} > 0$ for $i,j=1,2$.  We suppose that
  $u^\init, v^\init \ge 0$ are initial data which lie in $\fsW^{2,\infty}(\Omega)$ and are
  compatible with the homogeneous Dirichlet boundary condition \eqref{eq:bc-SKT}.

  Then, there exists a strong (that is, all terms appearing in the equation are
  defined a.e.) solution to system \eqref{eq:SKT}--\eqref{eq:bc-SKT} with initial
  data $u^\init, v^\init$.
\end{prop}

Previous results of global existence of strong solutions to \eqref{eq:SKT} with
Neumann boundary conditions were obtained in dimension $d=2$
\cite{louni1998,desvillettes2024a}, and in all dimension when the
cross-diffusion coefficient $\sigma$ is small \cite{choilui2003}, or when
self-diffusion is added \cite{choilui2004, tuoc2007, tuoc2008}. Dirichlet
boundary conditions were sometime considered as an easy adaptation of the proof.
\medskip

The rest of the paper is structured as follows: In Section~\ref{sec2}, the proof
of Theorem~\ref{thm:thm-dirichlet-bc} is presented. Then, short proofs of
Proposition~\ref{prord} and Proposition~\ref{thm:SKT-dir} are displayed in
Section~\ref{sec3}.

\section{Proof of Theorem \ref{thm:thm-dirichlet-bc}} \label{sec2}

For a domain \(\Omega\), let $\Gamma_{\Omega,D}(t, x, y)$ be the heat kernel on $\Omega$ with Dirichlet
boundary conditions, that is the solution to
\begin{equation}\label{eq:heat-kernel}
  \left\{
    \begin{aligned}
      &(\partial_t - \Delta_x) \Gamma_{\Omega,D}(t,x,y) = 0
      & & \text{for } (t,x) \in (0,\infty) \times \Omega, \\
      & \Gamma_{\Omega,D}(t,x,y) = 0
      & & \text{for } (t,x) \in (0,T) \times \partial\Omega, \\
      & \lim_{t \downarrow 0} \Gamma_{\Omega,D}(t,\cdot,y) = \delta_y
      & & \text{for } y\in\Omega.
    \end{aligned}
  \right.
\end{equation}

We first recall some classical bounds, see, e.g., \cite{boutondesvillettes2025}:

\begin{lemm}\label{thm:bound-gamma}
  Let $\Omega$ be a $\Ccal^2$ bounded domain of $\R^d$. For any time $T > 0$,
  $f \in \fsL^p([0,T]; \fsL^q(\Omega))$, with $p,q \in [1,\infty]$ satisfying
  $ \gamma := 2 - \frac{2}{p} - \frac{d}q > 0$ and $t\in [0,T]$, $x\in \Omega$, $A>0$, we
  have:
  \begin{equation*}
    \bigg| \int_0^t \int_{\Omega}  \Gamma_{\Omega,D} (\frac{t-s}A, x,y) \, f(s,y)\, \dd y\, \dd s \bigg|\lesssim_{A, T, d}
    t^{\frac \gamma 2} \|f\|_{\fsL^p([0,T]; \fsL^q(\Omega))} .
  \end{equation*}

  Moreover, for \(1/2>\epsilon>0\), $t\in [0,T]$ and $y\in \Omega$, we have:
  \begin{equation*}
    \int_{\Omega} |x-y|\, \Gamma_{\Omega,D}(t,x,y)\, \dd x \lesssim_{\eps, T, d, \Omega} t^{\frac 12 - \epsilon}.
  \end{equation*} 
\end{lemm}

In the case of Neumann boundary conditions, we needed an elaborate lemma (\cite [Lemma 13]{boutondesvillettes2025}) for the
lower bound on the heat kernel as the considered cylinders could intersect the
boundary of the domain. In the Dirichlet case, we only need to consider the
decay in the interior of the domain and thus the following simpler lemma suffices:

\begin{lemm}\label{thm:lower-bound-gamma}
  We consider the fundamental solution $\Gamma_{\B(0, R), D}$ of the heat equation
  with Dirichlet boundary conditions in $\B(0, R)$ for $R>0$.  Then for
  $T_R := a_0\frac{9R^2}{32d}$, the following estimate holds:
  \begin{equation*}
    \inf_{x, y \in \B(0, R/4)}
    \inf_{t \in \left[\frac{T_R}{2c_0a_0},\frac{T_R}{a_0}\right]}
    \Gamma_{\B(0, R), D}\, (t,x,y) \gtrsim_{c_0, d}R^{-d}.
  \end{equation*}
\end{lemm}

\begin{proof}
  We define
\begin{align*}
  p(t, |x|) &:= \frac{1}{(4\pi t)^{\frac d 2}}\exp\left(\frac{-|x|^2}{4t}\right),\\
  \Psi(t, x, y) &:= p(t, |x - y|) - \sup_{0 \ioe s\ioe t}p(s, 3R/4).
\end{align*}
Then:
\begin{itemize}
  \item if $x\in\partial\B(0, R), y\in\B(0, R/4)$, then $\Gamma_{\B(0, R), D}(t, x, y) = 0 \soe \Psi(t, x, y) = p(t, |x - y|) - \sup_{0 \ioe s\ioe t}p(s, 3R/4)$ since $|x - y|\soe 3R/4$ in this case;
  \item for $x\in\B(0, R)$, $y\in \B(0, R/4)$, $(\partial_t - \Delta_x)(\Gamma_{\B(0, R), D} - \Psi)(t, x, y) = \partial_t (\sup_{0 \ioe s\ioe t}p(s, 3R/4))\soe 0$ (this function is increasing).
\end{itemize}

Thus by the comparison principle in $\B(0, R)$ for a fixed $y\in \B(0, R/4)$, we have for $ t >0, x\in\B(0, R), y\in\B(0, R/4)$:
\begin{equation*}
  \Gamma_{\B(0, R), D}(t, x, y) \soe \Psi(t, x, y) = p(t, |x - y|) - \sup_{0 \ioe s\ioe t}p(s, 3R/4).
\end{equation*}

We follow the computations of the Neumann case in \cite{boutondesvillettes2025} and recall that:

\begin{equation*}
  \sup_{0 \ioe s\ioe t}p(s, |x|) =
    \begin{cases}
      p(t, |x|) &\text{ if } t \ioe \frac{|x|^2}{2d},\\
      p(\frac{|x|^2}{2d}, |x|) &\text { otherwise}.
    \end{cases}
\end{equation*}

Hence, if $t\ioe \frac{9R^2}{32d}$, then $\sup_{0 \ioe s\ioe t}p(s, \frac{3R}{4}) = p(t, \frac{3R}{4})$. But if $x, y\in\B(0, R/4)$ then $|x - y|\ioe R/2$, thus for $t\ioe \frac{9R^2}{32d}$:
\begin{align*}
  \Gamma_{\B(0, R), D}(t, x, y) &\soe p(t, R/2) - \sup_{0 \ioe s\ioe t}p(s, 3R/4)\\
  &= p(t, R/2) - p(t, 3R/4)\\
  &\soe \frac{1}{(4\pi t)^\frac{d}{2}}e^{-\frac{(\frac{3R}{4})^2}{4t}}\left(e^{\frac{5R^2}{16\cdot 4 t}}-1\right).
\end{align*}

If we now assume $0 < \frac{T_R}{2c_0a_0} \ioe t\ioe \frac{T_R}{a_0}\ioe\frac{9R^2}{32d}$, we see that

\begin{equation*}
  \Gamma_{\B(0, R), D}(t, x, y)\soe \left(\frac{a_0}{4\pi T_R}\right)^\frac{d}{2}e^{-\frac{c_0a_09R^2}{16\cdot 2T_R}}\frac{5a_0R^2}{16\cdot 4 T_R} .
\end{equation*}

Hence, if we take $T_R := a_0\frac{9R^2}{32d}$, we end up with the conclusion.
\end{proof}

The next step is to prove a decay of oscillation.  In the Neumann case of
\cite{boutondesvillettes2025}, we proved the decay of oscillations including the
neighbourhoods of the boundary.  Due to Dirichlet boundary conditions, we can
only prove a lemma in the interior of the domain.  Hence we arrive at the
following simpler proposition that is not concerned with the shape of the domain
\(\Omega\).

\begin{prop}\label{thm:reduction-oscillation-interior}
  Fix $d \in \N - \{0\}$, \(p, q \in [1, \infty] \) such that
  $ \gamma := 2 - \frac{2}{p} - \frac{d}q > 0$.  Set \(\beta = \frac{9a_0}{32d}\) and let
  $T>0$.  There exist a constant \(\delta>0\) depending only on $d$, $p$, $a_0$,
  $c_0$, and a constant $C_{f}$ that may also depend on $T$ such that for any
  $R>0$ satisfying $\beta R^2\ioe T$ and any function
  \(w : (-\beta R^2,0] \times B(0,R) \to [0,1]\) solving \eqref{eq:rough-scalar-pde} on
  $(-\beta R^2, 0]\times B(0, R)$ with a coefficient $a := a(t,x)$ fullfilling the bound
  \eqref{eq:assumption-a} over \((-\beta R^2,0] \times B(0, R)\) satisfies
  \begin{equation*}
    \oscillation_{(-\beta R^2/16,0] \times \B(0, R/4)} w  \le 1-\delta + C_f
    \, R^{\gamma} \, \| f \|_{\fsL^p((- \beta R^2,0); \fsL^{q}(B(0, R)))}.
  \end{equation*}
\end{prop}

\begin{proof}[Proof of Proposition~\ref{thm:reduction-oscillation-interior}]
  We have the initial time \(-T_R = -\beta R^2\) (matching
  \cref{thm:lower-bound-gamma}).

  We first assume that
  $|\{w(-T_R,\cdot) \ge \frac 12 \} \cap \B(0, R/4)| \ge \frac 12 |\B(0, R/4)|$, and
  consider $E:= \{w(-T_R,\cdot) \ge \frac 12\} \cap \B(0, R/4)$. We compare our solution
  with $v$ defined by:
\begin{equation*}
  \left\{
    \begin{aligned}
      & (a_0c_0\partial_t - \Delta)v = f
      & & \text{for } (t,x) \in (-T_R,0) \times \B(0, R), \\
      & v(t,x) = 0
      & & \text{for } (t,x) \in (-T_R,0) \times \partial \B(0, R), \\
      & v(-T_R,x) = \frac 12 \indic{E}(x)
      & & \text{for } x \in \B(0, R).
    \end{aligned}
  \right.
\end{equation*}
We notice that 
\begin{equation*}
  \left\{
    \begin{aligned}
     & (a_0c_0\partial_t - \Delta)(w - v) = (a_0c_0 -a)\partial_t w \soe 0 & & \text{for } (t,x) \in (-T_R,0) \times \B(0, R), \\
     & (w - v)(t, x) = w(t, x) \soe 0 & & \text{for } (t,x) \in (-T_R,0) \times \partial \B(0, R).
    \end{aligned}
  \right.
\end{equation*}

By the comparison principle, we get for $\frac{-T_R}{2} \ioe t\ioe 0$, $x\in\B(0, R/4)$:
\begin{align*}
  w(t, x)
  &\soe v(t, x)\\
  &= \int_{y \in \B(0, R)} \Gamma_{\B(0, R), D}(\frac{t+T_R}{a_0c_0},x,y) \,\frac 12 \indic{E}(y)\,
    \dd y\\
  &\phantom{=} + \int_{s=-T_R}^t \int_{y \in \B(0, R)}
    \Gamma_{\B(0, R), D}(\frac{t-s}{a_0c_0},x,y) f(s,y)\, \dd y\, \dd s\\
  &\gtrsim_{c_0, d} R^{-d}\frac{|\B(0, R/4)|}{2} \\
  &\phantom{=} -\int_{\sigma=0}^{t+T_R} \int_{y \in \B(0, R)}
    \Gamma_{\B(0, R), D}\left(\frac{(t+T_R)-\sigma}{a_0c_0},x,y\right) |f(\sigma-T_R,y)|\, \dd y\, \dd \sigma\\
  &\soe \delta - C_1 (t+T_R)^{\frac \gamma 2}\|f\|_{\fsL^{p}([-T_R,0]; \fsL^{q}(\B(0, R)))}\\
  &\soe \delta - C_2 R^{\gamma}\|f\|_{\fsL^{p}([-T_R,0]; \fsL^{q}(\B(0, R)))},
\end{align*}
for $\delta$ some constant depending only on $a_0, c_0, p, q$ and $C_1, C_2$ that may
also depend on \(T\).  Here we used in the second inequality
Lemma~\ref{thm:lower-bound-gamma} to get a positive term and
Lemma~\ref{thm:bound-gamma} for bounding the forcing term in the next
inequality.  \medskip

We now treat the other case $|\{w(-T_R,\cdot) \ioe \frac 12 \} \cap \B(0, R/4)| \ge \frac 12 |\B(0, R/4)|$ by considering
 again $E:= \{w(-T_R,\cdot) \ioe \frac 12\} \cap \B(0, R/4)$. We compare $1 - w$ with $v$ defined by

\begin{equation*}
  \left\{
    \begin{aligned}
      & (a_0\partial_t - \Delta)v = f
      & & \text{for } (t,x) \in (-T_R,0) \times \B(0, R), \\
      & v(t,x) = 0
      & & \text{for } (t,x) \in (-T_R,0) \times \partial \B(0, R), \\
      & v(-T_R,x) = \frac 12 \indic{E}(x)
      & & \text{for } x \in \B(0, R).
    \end{aligned}
  \right.
\end{equation*}

Since we assumed that $w\ioe 1$ in $\B(0, R)\subset \Omega$, we have $1- w \soe 0 = v$ on $\partial\B(0, R)$, thus we similarly end up with $1 - w \soe v$. From this point, the proof is the same as in the first case.
\end{proof}

Before proving Theorem~\ref{thm:thm-dirichlet-bc}, we present two intermediate lemmas that allow us to bound $w(t, z)$ for $z$ near the boundary $\partial\Omega$ or for small times $t$. We introduce first the following notation:
\medskip

For $c > 0$, we define $v_{c}$ as the solution to
  \begin{equation}\label{eq:def-vc}
    \left\{
      \begin{aligned}
        &(c \partial_t - \Delta) v_{c}(t,x) = f(t,x)
        & & \text{for } (t,x) \in (0,T] \times \Omega,\\
        & v_{c}(t,x) = 0
        & & \text{for } (t,x) \in (0,T] \times \partial\Omega,\\
        & v_{c}(0,x) = w_{\init}(x)
        & & \text{for } x \in \Omega.
      \end{aligned}
    \right.
  \end{equation}

  Within the framework of Theorem~\ref{thm:thm-dirichlet-bc}, we observe that the comparison principle yields 
the inequality
  \begin{equation}\label{eq:bound-vc}
  v_{a_0c_0}\ioe w\ioe v_{a_0} \quad\text{on } \quad\Omega_T.
  \end{equation}
  Indeed, as $a_0 \ioe a$ and $\partial_t w \soe 0$, we find
  \begin{equation*}
    \left\{
      \begin{aligned}
        &(a_0\partial_t - \Delta)(v_{a_0} - w) = - (a_0 - a)\partial_t w \soe 0 & &\text{in}\quad\Omega_T,\\
        & v_{a_0} - w = 0& &\text{on} \quad (0, T]\times\partial\Omega.
      \end{aligned}
    \right.
  \end{equation*}
  This implies by the comparison principle that $w\ioe v_{a_0}$, and the other inequality in \eqref{eq:bound-vc} follows similarly.

\begin{lemm}\label{thm:boundary-bound}
  Within the framework of Theorem~\ref{thm:thm-dirichlet-bc}, let
  $\tilde{\gamma} := \min(\gamma, 1-2\eps)$, where $\eps >0$ is a small parameter. For
  $z\in\Omega$, we set the distance $d_z := d(z, \partial\Omega)$. Then there exists some constant
  $C$ that may depend on $\Omega, \gamma, a_0, c_0, d, T$ and the small parameter $\eps$
  such that for all $x\in\Omega$, $t\in[0, T]$:
  \begin{equation}
    w(t, x) \ioe C\, (\|w_\init\|_{\Lip(\overline\Omega)} + \|f\|_{\fsL^p([0, T], \fsL^q(\Omega))})\, d_x^{\tilde{\gamma}}.
  \end{equation}
\end{lemm}

\begin{proof}
  We have to consider two different cases, depending on how large $t$ is.\medskip

  First we assume that $t^{\frac 12} \ioe d_x$.  For a point
  $z\in \Omega$, the distance to the boundary is bounded by the triangular inequality
  as $d_z \ioe |z - x| + d_x$.  Hence there exists $z_0\in\partial\Omega$ with
  $|z - z_0| = d_z \le |z-x| + d_x$, which implies
  $w_\init(z) \ioe |w_\init(z) - w_\init(z_0)| \ioe (|z - x| + d_x)\|w_\init\|_{\Lip}$. We
  consider $\Gamma_{\Omega, D}$ the heat kernel with Dirichlet boundary conditions in $\Omega$
  and we recall \eqref{eq:def-vc} and \eqref{eq:bound-vc}. Then,
  \begin{align*}
    w(t, x) &\ioe v_{a_0}(t, x) \\
    &= \int_\Omega \Gamma_{\Omega, D}(\frac{t}{a_0}, x, z)w_\init(z) + \int_0^t\int_\Omega \Gamma_{\Omega, D}(\frac{t-s}{a_0}, x, z)f(s, z)\\
    &\lesssim \|w_\init\|_{\Lip}\int_\Omega \Gamma_{\Omega, D}(\frac{t}{a_0}, x, z)(|z-x| + d_x) + t^{\frac{\gamma}{2}} \|f\|_{\fsL^p([0, T], \fsL^q(\Omega))}\\
    &\lesssim \|w_\init\|_{\Lip}(t^{\frac{1}{2}-\eps} + d_x) + t^{\frac{\gamma}{2}} \|f\|_{\fsL^p([0, T], \fsL^q(\Omega))}\\
    &\lesssim (\|w_\init\|_{\Lip}+\|f\|_{\fsL^p([0, T], \fsL^q(\Omega))})\, d_x^{\tilde{\gamma}}.
  \end{align*}

  To pass from the second to the third line, we used Lemma~\ref{thm:bound-gamma} (first inequality) to bound the second term and we used the same lemma (second inequality) to pass from the third to the fourth line for the first term.\medskip

  We now consider the case when $t$ is large: $t \soe d_x^2$. We need to
  understand the behaviour of $\Gamma_{\Omega, D}$ near the boundary for large $t$. We
  then compare $w$ and $v_{a_0}$. We use for that some Gaussian estimates near
  the boundary as in \cite{hui1992}, where it is claimed that for some $c > 0$:
  \begin{equation*}
    \Gamma_{\Omega, D}(t, x, y) \lesssim \frac{d_xd_y}{t^{\frac{d+2}{2}}}e^{-c\frac{|x - y|^2}{t}}.
  \end{equation*}
  
  We first notice that
  \begin{align*}
    w(t, x) &\ioe v_{a_0}(t, x)\\
    &= \int_\Omega \Gamma_{\Omega, D}(\frac{t}{a_0}, x, z)w_\init(z)\dd z + \int_0^t\int_\Omega \Gamma_{\Omega, D}(\frac{t-s}{a_0}, x, z)f(s, z)\dd z\dd s \\
    &= \int_\Omega \Gamma_{\Omega, D}(\frac{t}{a_0}, x, z)w_\init(z)\dd z\\
    &\phantom{=} + \int_0^{t-d_x^2}\int_\Omega \Gamma_{\Omega, D}(\frac{t-s}{a_0}, x, z)f(s, z)\dd z \dd s\\
    &\phantom{=} + \int_{t-d_x^2}^t\int_\Omega \Gamma_{\Omega, D}(\frac{t-s}{a_0}, x, z)f(s, z)\dd z \dd s.
  \end{align*}

  We will bound the three terms appearing in the previous expression. Similarly as in the first case, we have $d_z\ioe |z - x| + d_x$, and hence $w_\init(z) \ioe (|z - x| + d_x) \|w_\init\|_{\Lip}$. We obtain
  \begin{align*}
    \int_\Omega \Gamma_{\Omega, D}(\frac{t}{a_0}, x, z)w_\init(z) \dd z&\lesssim \|w_\init\|_{\Lip}\frac{d_x}{t^{\frac d2 + 1}}\int_\Omega e^{-ca_0\frac{|x-z|^2}{t}}(|z - x| + d_x)^2\dd z \\
    &\lesssim \|w_\init\|_{\Lip}\frac{d_x}{t^{\frac d2 + 1}}\int_{\R^d} e^{-ca_0\frac{|x-z|^2}{t}}(|z - x|^2+d_x^2)\dd z \\
    &\lesssim \|w_\init\|_{\Lip}\frac{d_x}{t^{\frac d2 + 1}}(t^{\frac d2 +1}+t^{\frac d2}d_x^2)\\
    &\lesssim \|w_\init\|_{\Lip}d_x(1 + \frac{d_x}{t^\frac12})\\
    &\lesssim \|w_\init\|_{\Lip}d_x,
  \end{align*}
  where to obtain the last line, we used the assumption $d_x\ioe t^\frac12$.\medskip

  Then by Lemma~\ref{thm:bound-gamma}, we have immediately
  \begin{align*}
    \int_{t-d_x^2}^t\int_\Omega \Gamma_{\Omega, D}(\frac{t-s}{a_0}, x, z)f(s, z)\dd z \dd s
    &= \int_0^{d_x^2}\int_\Omega \Gamma_{\Omega, D}(\frac{s}{a_0}, x, z)f(t- s, z)\dd z \dd s\\
    &\lesssim d_x^\gamma\, \|f\|_{\fsL^p([0, T], \fsL^q(\Omega))}.
  \end{align*}

  Finally (we note that $\gamma > 0$ implies $p>1$):
  \begin{align*}
    &\int_0^{t-d_x^2}\int_\Omega \Gamma_{\Omega, D}(\frac{t-s}{a_0}, x, z)f(s, z)\dd z \dd s\\
    &\lesssim d_x\int_0^{t-d_x^2}\frac{1}{(t-s)^{\frac d2 + 1}}\int_\Omega e^{-ca_0\frac{|x-z|^2}{t-s}}(|z - x|+d_x)f(s, z)\dd z \dd s\\
    &\lesssim d_x\int_0^{t-d_x^2}\frac{1}{(t-s)^{\frac d2 + 1}}\|f(s)\|_{\fsL^q(\Omega)} \|z \mapsto e^{-ca_0\frac{|x-z|^2}{t-s}}(|z - x|+d_x)\|_{\fsL^{q'}(\R^d)}\dd s\\
    &\lesssim d_x\int_0^{t-d_x^2}\frac{1}{(t-s)^{\frac d2 + 1}}\|f(s)\|_{\fsL^q(\Omega)} \left((t-s)^{\frac{d}{2q'} + \frac 12} + (t-s)^{\frac{d}{2q'}}d_x\right)\dd s\\
    &\lesssim d_x\|f\|_{\fsL^p([0, T], \fsL^q(\Omega))}\left[\left(\int_{d_x^2}^t s^{(- \frac d2 - 1 + \frac{d}{2q'} + \frac 12)p'}\dd s\right)^{\frac{1}{p'}} + d_x\left(\int_{d_x^2}^t s^{(- \frac d2 - 1 + \frac{d}{2q'})p'}\dd s\right)^{\frac{1}{p'}}\right]\\
    &\lesssim d_x\|f\|_{\fsL^p([0, T], \fsL^q(\Omega))}\left[\left(\int_{d_x^2}^t s^{p'\frac{\gamma - 1}{2} - 1}\dd s\right)^{\frac{1}{p'}} + d_x\left(\int_{d_x^2}^t s^{p'\frac{\gamma - 2}{2} - 1}\dd s\right)^{\frac{1}{p'}}\right].
  \end{align*}

  To obtain the last line we used $(- \frac d2 - 1 + \frac{d}{2q'} + \frac 12)p' = p'\frac{\gamma - 1}{2} - 1$ and $(- \frac d2 - 1 + \frac{d}{2q'})p' = p'\frac{\gamma - 2}{2} - 1$.\medskip
  
  We recall that $0<\gamma\ioe2$, and first assume that $\gamma < 2$. Then
  \begin{equation*}
    d_x^2\left(\int_{d_x^2}^t s^{p'\frac{\gamma - 2}{2} - 1}\dd s\right)^{\frac{1}{p'}}\lesssim d_x^2\,
d_x^{\gamma - 2}\ioe d_x^\gamma.
  \end{equation*}

  If $\gamma =2$ (which amounts to $p=q=+\infty$ and thus $p'=1$)
  \begin{equation*}
    d_x^2\,\int_{d_x^2}^t s^{- 1}\dd s\ioe d_x^2\, (2\lvert \ln d_x\rvert +\lvert\ln T\rvert)\lesssim_T d_x^{\tilde{\gamma}}.
  \end{equation*}

  Similarly, if $\gamma < 1$, we get
  \begin{equation*}
    d_x\left(\int_{d_x^2}^t s^{p'\frac{\gamma - 1}{2} - 1}\dd s\right)^{\frac{1}{p'}}\lesssim d_x\, d_x^{\gamma - 1}\ioe d_x^\gamma.
  \end{equation*}

  If $\gamma = 1$, we obtain
  \begin{equation*}
    d_x\left(\int_{d_x^2}^t s^{p'\frac{\gamma - 1}{2} - 1}\dd s\right)^{\frac{1}{p'}}\ioe d_x\, (|2\ln d_x|^{\frac{1}{p'}} + \lvert\ln T\rvert)\lesssim_{T, p} d_x^{\tilde{\gamma}}.
  \end{equation*}

  Finally, if $\gamma >1$, we see that
    \begin{equation*}
    d_x\left(\int_{d_x^2}^t s^{p'\frac{\gamma - 1}{2} - 1}\dd s\right)^{\frac{1}{p'}}\ioe d_x \, T^{\frac{\gamma - 1}{2}}\lesssim_{T, p} d_x.
  \end{equation*}

  Summarizing all the estimates above, we have shown that

  \begin{equation*}
    w(t, x) \lesssim (\|w_\init\|_{\Lip} + \|f\|_{\fsL^p([0, T], \fsL^q(\Omega))}) \, d_x^{\tilde{\gamma}}.
  \end{equation*}

  This concludes the second case and thus the proof of the lemma.
\end{proof}

\begin{lemm}\label{thm:short-time-bond}
  Within the framework of Theorem~\ref{thm:thm-dirichlet-bc}, there exists some
  constant $C$ that may depend on $\Omega, \gamma, a_0, c_0, d, T$ such that for all
  $x\in\Omega$, $t\in[0, T]$,
  \begin{equation*}
    |w(t, x)-w(0, x)| \ioe C \, (\|w_\init\|_{\Ccal^1(\overline\Omega)} + \|f\|_{\fsL^p([0, T], \fsL^q(\Omega))}) \, t^{\frac{\min(1, \gamma)}{2}}.
  \end{equation*}
\end{lemm}

\begin{proof}
  To prove this estimate, we cannot use the same method as in the Neumann case
  of \cite{boutondesvillettes2025}, because the mass $\int\Gamma_{\Omega, D}(t, x, y)\, \dd y$
  is not preserved anymore.  Nevertheless, using \eqref{eq:bound-vc}, we know
  that
  $$v_{a_0c_0}(t, x) - w(0, x)\ioe w(t, x) - w(0, x)\ioe v_{a_0}(t, x)- w(0, x).$$

  Now we use the decomposition $v_{a_0} = v_{a_0, f} + \tilde{v}_{a_0}$ with
  \begin{equation*}
    \left\{
      \begin{aligned}
        &(a_0 \partial_t - \Delta) v_{a_0, f}(t,x) = f(t,x)
        & & \text{for } (t,x) \in (0,T] \times \Omega,\\
        & v_{a_0, f}(t,x) = 0
        & & \text{for } (t,x) \in (0,T] \times \partial\Omega,\\
        & v_{a_0, f}(0,x) = 0
        & & \text{for } x \in \Omega,
      \end{aligned}
    \right.
  \end{equation*}
  and

  \begin{equation*}
    \left\{
      \begin{aligned}
        &(a_0 \partial_t - \Delta) \tilde{v}_{a_0}(t,x) = 0
        & & \text{for } (t,x) \in (0,T] \times \Omega,\\
        & \tilde{v}_{a_0}(t,x) = 0
        & & \text{for } (t,x) \in (0,T] \times \partial\Omega,\\
        & \tilde{v}_{a_0}(0,x) = w_{\init}(x)
        & & \text{for } x \in \Omega.
      \end{aligned}
    \right.
  \end{equation*}

 Lemma~\ref{thm:bound-gamma} provides the bound
  \begin{equation*}
    |v_{a_0, f}(t,x)|\lesssim t^{\frac{\gamma}{2}}\|f \|_{\fsL^p([0,T]; \fsL^q(\Omega))}.
  \end{equation*}

  Using a Schauder estimate proven in \cite[Theorem~5.1.11]{lunardi1995}, we get
  the inequality
  $\|\tilde{v}_{a_0}\|_{\Ccal^{0, 1}_{\text{par}}([0, T] \times\overline\Omega)} \lesssim \|w_\init\|_{\Ccal^1(\overline\Omega)}$,
  where
  \begin{equation*}
  \Ccal^{0, 1}_{\text{par}} := \left\{\fctset{u}{[0, T]\times\Omega}{\R} : \exists C, \forall (t,x)\neq (t', y)\in\Omega\times[0, T], \frac{|u(t, x) - u(t', y)|}{|t-t'|^\frac12 + |x - y|}\ioe C\right\}
  \end{equation*}
  is a parabolic Hölder space.  Hence
  $|\tilde{v}_{a_0}(t, x) - \tilde{v}_{a_0}(0, x)|\lesssim t^\frac{1}{2}\|w_\init\|_{\Ccal^1(\overline\Omega)}$,
  and
  \begin{align*}
    w(t, x)- w(0, x)
    &\ioe v_{a_0}(t, x)- w(0, x) \\
    &\ioe \tilde{v}_{a_0}(t, x) - w(0, x) + v_{a_0, f}(t,x)\\
    &\lesssim (t^{\frac{1}{2}} + t^{\frac{\gamma}{2}})(\|w_\init\|_{\Ccal^1(\overline\Omega)} + \|f \|_{\fsL^p([0,T]; \fsL^q(\Omega))}).
  \end{align*}
  
  Similarly
  \begin{align*}
    w(t, x)- w(0, x) 
    &\soe \tilde{v}_{a_0c_0}(t, x) - w(0, x) + v_{a_0c_0, f}(t,x)\\
    &\gtrsim - (t^{\frac{1}{2}} + t^{\frac{\gamma}{2}})(\|w_\init\|_{\Ccal^1(\overline\Omega)} + \|f \|_{\fsL^p([0,T]; \fsL^q(\Omega))}).
      \qedhere
  \end{align*}
\end{proof}

We now can conclude the 

\begin{proof}[Proof of Theorem~\ref{thm:thm-dirichlet-bc}]
  Let $x, y\in\Omega$, $0 < t', t < T$ and let $\eps > 0$ be a small parameter. We
  recall that we define for $z\in\Omega$ the distance $d_z := d(z, \partial\Omega)$.  All the
  constants that we write under $\lesssim$ are independent of $x, y, t', t, d_x, d_y$
  but may depend on $\Omega, T, a_0, c_0, d$ and the small parameter $\eps$. We
  recall that $\tilde{\gamma} := \min(\gamma, 1-2\eps)$.  \medskip

  \textit{First case; near the boundary: $\min(d_x, d_y)\ioe 2|x-y|$.}\newline
  We notice that $d_y \ioe d_x + |x - y|$ and $d_x \ioe d_y + |x - y|$, hence $d_x, d_y \ioe 3|x-y|$. We can thus use Lemma~\ref{thm:boundary-bound} for $(t, x)$ and $(t', y)$ to obtain:
  \begin{equation*}
    |w(t, x) - w(t', y)| \ioe |w(t, x)| + |w(t', y)| \lesssim (\|w_\init\|_{\Lip} + \|f\|_{\fsL^p([0, T], \fsL^q(\Omega))})|x-y|^{\tilde\gamma}.
  \end{equation*}

  This yields the desired Hölder regularity for $\a := \tilde{\gamma}$ in the case when $\min(d_x, d_y)\ioe 2|x-y|$.
  \medskip

  \textit{Second case; in small time: $t'^{\frac 12}, t^{\frac 12} \ioe |x - y|$.}\newline
  We use Lemma~\ref{thm:short-time-bond} with $(t, x)$ and $(t', y)$, to obtain:
   \begin{align*}
    &|w(t, x) - w(t', y)| \nonumber\\ 
    &\ioe |w(0, x) - w(0, y)|+ |w(t, x) - w(0, x)|+ |w(t', y) - w(0, y)|\\
    &\lesssim |x - y|\|w_\init\|_{\Ccal^1(\overline\Omega)} +  (\|w_\init\|_{\Ccal^1(\overline\Omega)} + \|f\|_{\fsL^p([0, T], \fsL^q(\Omega))})(t^{\frac{\tilde{\gamma}}{2}} + t'^{\frac{\tilde{\gamma}}{2}})\\
    &\lesssim (\|w_\init\|_{\Ccal^1(\overline\Omega)} + \|f\|_{\fsL^p([0, T], \fsL^q(\Omega))})|x - y|^{\tilde{\gamma}}.
  \end{align*}
    This yields the desired Hölder regularity for $\a := \tilde{\gamma}$ in the case when $t'^{\frac 12}, t^{\frac 12} \ioe |x - y|$.\medskip

    \textit{Third case; far from the boundary and in large time:
      $|x - y| \ioe \min(t'^{\frac 12}, t^{\frac 12}, \frac{\min(d_x, d_y)}{2})$}\newline
    Without loss of generality, we can assume that $t'\ioe t$.

    We use the reduction of oscillations, starting from the scale
    $R := \min(t^{\frac 12}, \min(d_x, d_y))$. Let $z_1, z_2\in\B(x, R)$,
    $\tilde t_1, \tilde t_2 \in (t - \beta R^2, t]$.\medskip
  
    If $\min(d_x, d_y)\ioe t^{\frac 12}$, we notice that
    $d_x \ioe d_y + |x - y|$, thus
    $d_x \ioe \min(d_x, d_y) + |x - y| \ioe \frac32 \min(d_x, d_y) \lesssim R$. Hence,
    we obtain that $d_{z_1}, d_{z_2} \ioe d_x + R \lesssim R$.  Thus we can use
    Lemma~\ref{thm:boundary-bound}, as in the first case, to obtain
  \begin{align*}
    w(\tilde t_1, z_1), w(\tilde t_2, z_2)
    &\lesssim (\|w_\init\|_\Lip + \|f \|_{\fsL^p([0,T]; \fsL^q(\Omega))}) R^{\tilde{\gamma}}.
  \end{align*}
  From this, we conclude
  \begin{align*}
    \oscillation_{(t-\beta R^2,t] \times B(x,R )} w
    &\ioe 2\|w\|_{\fsL^\infty((t-\beta R^2,t] \times B(x,R ))}\\
    &\lesssim (\|w_\init\|_\Lip + \|f \|_{\fsL^p([0,T]; \fsL^q(\Omega))}) R^{\tilde{\gamma}}.
  \end{align*}
  
  If $\min(d_x, d_y)\soe t^{\frac 12}$, we notice that $\tilde t_1\ioe t \ioe R^2$, thus we can use Lemma~\ref{thm:short-time-bond}, as in the second case, to obtain
  \begin{align*}
    |w(\tilde t_1, z_1) - w(0, z_1)|
    &\lesssim (\|w_\init\|_{\Ccal^1(\overline\Omega)} + \|f \|_{\fsL^p([0,T]; \fsL^q(\Omega))}) {\tilde{t}_1}^{\frac{\tilde{\gamma}}{2}}\\
    &\lesssim (\|w_\init\|_{\Ccal^1(\overline\Omega)} + \|f \|_{\fsL^p([0,T]; \fsL^q(\Omega))})R^{\tilde{\gamma}}.  
  \end{align*}
  Similarly,
  $$|w(\tilde t_2, z_2) - w(0, z_2)|\lesssim (\|w_\init\|_{\Ccal^1(\overline\Omega)} + \|f \|_{\fsL^p([0,T]; \fsL^q(\Omega))})R^{\tilde{\gamma}}.$$
  Thus,
  \begin{align*}
    &|w(\tilde t_1, z_1) - w(\tilde t_2, z_2)|\\
    &\ioe |w(\tilde t_1, z_1) - w(0, z_1)| + |w(\tilde t_2, z_2) - w(0, z_2)| + |w(0, z_1) - w(0, z_2)|\\
    &\lesssim (\|w_\init\|_{\Ccal^1(\overline\Omega)} + \|f \|_{\fsL^p([0,T]; \fsL^q(\Omega))}) R^{\tilde{\gamma}}.
  \end{align*}

  Hence in both cases:
  \begin{equation*}
    \oscillation_{(t-\beta R^2,t] \times B(x,R )} w \ioe C(\|w_\init\|_{\Ccal^1(\overline\Omega)} + \|f \|_{\fsL^p([0,T]; \fsL^q(\Omega))}) R^{\tilde{\gamma}},
  \end{equation*}
  for some constant $C$ that does not depend on $t$ or $d_x, d_y$.\medskip
  
  Let $o_k := \oscillation_{(t- \beta R^2 4^{-2{k}},t] \times B(x,R 4^{-{k}})} w$,
  $m_k := \inf_{(t- \beta R^2 4^{-2{k}},t] \times B(x,R 4^{-{k}})} w$ and
  $\tilde{C}_f := C_fR^\gamma\| f \|_{\fsL^p([0, T]; \fsL^{q}(\Omega))}$, where $C_f$ is
  defined in Proposition~\ref{thm:reduction-oscillation-interior}. We notice
  that
  $w_k(\tilde t, \tilde x) := \frac{w(t + \tilde t, x + \tilde x) - m_k}{o_k} \in [0,1]$
  for any
  $\tilde t \in (-\beta R^24^{-2k}, 0], \tilde{x}\in\B(0, R4^{-k}) \subset \Omega$. Furthermore,
  $w_k$ is a solution to
\begin{equation*}
  a(t + \tilde t,x + \tilde x) \partial_{\tilde t} w_k(\tilde t,\tilde x) - \Delta_{\tilde x} w_k(\tilde t, \tilde x) = \frac{f(t+ \tilde t, x + \tilde x)}{o_k}
  \qquad \text{and}  \quad \partial_{\tilde t} w_k \ge 0.
\end{equation*}
  The function $(\tilde t, \tilde x) \mapsto a(t + \tilde t,x + \tilde x)$ satisfies the bound \eqref{eq:assumption-a}, thus one can apply Proposition~\ref{thm:reduction-oscillation-interior} (at scale $R4^{-k}$) to $w_k$:
  \begin{align*}
    o_{k+1} &= o_k \oscillation_{(t- \beta R^2 4^{-2(k+1)},t] \times B(x,R 4^{-(k+1)})} \left(\frac{w - m_k}{o_k}\right)\\
    &\ioe (1-\delta) o_k + C_f \,\, 
    \, \left(\frac{R}{4^k}\right)^{\gamma} \, \| f \|_{\fsL^p((0, T); \fsL^{q}(\Omega))}\\
    & = (1 - \delta)o_k + \tilde{C}_f4^{-\gamma k}.
  \end{align*}

  Hence, we obtain that
  \begin{align*}
    o_k
    &\ioe \sum_{l= 0}^{k-1} (1-\delta)^l4^{-((k-1)-l)\gamma}\tilde{C}_f + (1 - \delta)^ko_0\\
    &\ioe \sum_{l= 0}^{k-1} \max(\frac{1}{4^\gamma}, (1 - \delta))^{k-1}\tilde{C}_f + (1 - \delta)^ko_0\\
    &\ioe \tilde{C}_f\, k\max(\frac{1}{4^\gamma}, (1 - \delta))^{k-1}+ (1 - \delta)^ko_0\\
    &\ioe (\tilde{C}_f+o_0)\, k\,\max\left(\frac{1}{4^\gamma}, (1 - \delta) \right)^{k-1}\\
    &\lesssim R^{\tilde{\gamma}}\,(\| f \|_{\fsL^p((0, T); \fsL^{q}(\Omega))}+\|w_\init\|_{\Ccal^1(\overline\Omega)})\,k\,\max\left(\frac{1}{4^\gamma}, (1 - \delta) \right)^{k-1}.
  \end{align*}

  Let $\Lambda\soe\max(\frac{1}{4^\gamma}, (1 - \delta)) + \eta$ for some small $\eta > 0$ such that $\Lambda < 1$ ($\eta$ only depends on $\delta, \gamma$). Then
  \begin{equation*}
      o_k \lesssim_{\Lambda, \eta} R^{\tilde{\gamma}}(\| f \|_{\fsL^p((0, T); \fsL^{q}(\Omega))}+\|w_\init\|_{\Ccal^1(\overline\Omega)})\Lambda^k.
  \end{equation*}
  For $(x, t) \neq (y, t')$, let $k_0$ be the smallest integer such that: $R4^{-k_0} \ioe \max(|x - y|, (\frac{|t - t'|}{\beta})^\frac 12)$. Then:
  \begin{align*}
    |w(t, x) - w(t', y)| &\ioe o_{k_0-1}\\
    &\lesssim \Lambda^{k_0-1}R^{\tilde\gamma}(\| f \|_{\fsL^p((0, T); \fsL^{q}(\Omega))}+\|w_\init\|_{\Ccal^1(\overline\Omega)}).
  \end{align*}

  Letting $\a_\Lambda := \frac{-\ln \Lambda}{\ln 4}$, we find
  \begin{align*}
    \Lambda^{k_0-1}R^{\tilde\gamma}
    &= \exp((k_0-1)\alpha_\Lambda(-\ln4))R^{\tilde\gamma}\\
    &\ioe 4^{-\a_\Lambda(k_0 - 1)}R^{\tilde\gamma}\\
    &\lesssim \max(|x - y|, |t - t'|^\frac 12)^{\a_\Lambda}R^{\tilde\gamma - \a_\Lambda}\\
    &\ioe \max(|x - y|, |t - t'|^\frac 12)^{\a_\Lambda}T^{\frac{\tilde\gamma - \a_\Lambda}{2}}.
  \end{align*}

  The last line holds provided that $\a_\Lambda\ioe \tilde\gamma$, which we can always ensure by increasing $\Lambda$ close to $1$. Thus:
  \begin{equation*}
    |w(t, x) - w(t', y)| \lesssim (|x - y|^{\a_\Lambda} +  |t - t'|^\frac{\a_\Lambda}{2})(\| f \|_{\fsL^p((0, T); \fsL^{q}(\Omega))}+\|w_\init\|_{\Ccal^1(\overline\Omega)}).
  \end{equation*}

  This yields the desired Hölder regularity and thus concludes the proof.
\end{proof}

\section{Applications}\label{sec3}
\subsection{Application to reversible chemistry}\label{sec:chemistry}

\begin{proof}[Proof of Proposition~\ref{prord}]
  We only sketch the proof, as it is very close to that of
  \cite{boutondesvillettes2025}. The most standard a priori estimate, which
  works for the Dirichlet as well as the Neumann boundary conditions, states
  (cf.\ \cite{boutondesvillettes2025}) that for some constant $C_p >0$ depending
  only on $p>0$, and for all $T>0$,
  \begin{multline}
    \label{u14p}
    \sum_{i=1}^4 \int_{\Omega} \frac{u_i^{p+1}}{p+1} (T)
    + \frac{4p}{(p+1)^2}  \sum_{i=1}^4  d_i \int_0^T \int_{\Omega} |\nabla ( u_i^{\frac{p+1}2} )|^2\\
    \le  \sum_{i=1}^4 \int_{\Omega} \frac{u_i^{p+1}}{p+1} (0) + C_p \, \sum_{i=1}^4  \int_0^T\int_{\Omega} {u_i^{p+2}} .
  \end{multline}

  We define $w:= \int_0^t\left(\sum_{i=1}^4 d_i u_i\right)$, and check that
  $w \soe 0$, $\partial_t w \soe 0$, $w(t, x) = 0$ if $x\in\partial\Omega$. Moreover
\begin{equation*}
  \Delta w = \int_0^t\Delta\left(\sum_{i=1}^4 d_i u_i\right) = \sum_{i=1}^4 \int_0^t\partial_tu_i =\sum_{i=1}^4u_i - \sum_{i=1}^4u_i^\init.
\end{equation*}
Thus, defining $a := \frac{\sum_{i=1}^4u_i}{\sum_{i=1}^4d_1u_i}$, an immediate computation yields
\begin{equation*}
  a\partial_t w - \Delta w = \sum_{i=1}^4u_i^\init,
\end{equation*}
while the bound $\frac{1}{\max d_i}\ioe a \ioe \frac{1}{\min d_i}$ clearly holds. Theorem~\ref{thm:thm-dirichlet-bc} implies
then that $\sum_{i=1}^4u_i = \sum_{i=1}^4u_i^\init+ \Delta w$, with $w\in\Ccal^{0,\a}([0,T] \times \overline{\Omega})$ for some $\a>0$.
\medskip

Then we use the following interpolation identity, used at a given time $t$ (cf. \cite{boutondesvillettes2025}),
 \begin{equation} \label{esalp2}
    \bigg\|  \sum_{i=1}^4  u_i  -  \sum_{i=1}^4  u_i^\init \bigg\|_{\fsL^{2\,\frac{3 - \alpha}{2-\alpha}} (\Omega)}^3
    \le  C\, \| {w} \|_{\Ccal^{0,\alpha}(\overline{\Omega})}^{\frac3{3-\alpha}}\, \bigg\| \nabla \bigg[  \sum_{i=1}^4  u_i  -  \sum_{i=1}^4  u_i^\init \bigg]\, \bigg\|_{\fsL^2(\Omega)}^{3\frac{2 -\alpha}{3-\alpha}} .
  \end{equation}
  Using estimate \eqref{u14p} for $p=1$ together with estimate \eqref{esalp2}
  ensures (cf.\ \cite{boutondesvillettes2025}) that $u_i \in \fsL^{3+\delta}([0,T] \times \Omega)$
  for some $\delta>0$, which is sufficient to get an estimate (for $u_i$) in all
  $\fsL^p([0,T] \times \Omega)$ when $p\in [1, \infty[)$. With this estimate, existence of strong
  solutions is obtained after the use of an approximation process (still
  cf. \cite{boutondesvillettes2025}).
\end{proof}

\subsection{Application to the triangular SKT system}\label{SKT}

\begin{proof} [Proof of Proposition~\ref{thm:SKT-dir}]

  Here also, we only sketch the proof, since it is very close to
  \cite{boutondesvillettes2025}.  The standard a priori estimate, which works
  for the Dirichlet as well as the Neumann boundary conditions, states (cf.\
  \cite{boutondesvillettes2025}) that for some constant $C_p >0$ depending only
  on $p>0$, and for all $T>0$,
  \begin{equation} \label{eqp}
    \int_{\Omega} \frac{u^{p+1}}{p+1} (T) + d_1\,\frac{4p}{(p+1)^2} \, \int_0^T \int_{\Omega} \,|\nabla (u^{\frac{p+1}{2}})|^2
    \le  C_p + C_p \int_0^T \int_{\Omega} u^{p+2} .
  \end{equation}

  An application of the maximum principle also shows that
  $\|v\|_{\fsL^\infty(\Omega_T)} \ioe C$.\medskip

We introduce the quantity $m:= m(t,x)$ defined as the solution of the heat equation
$$ \partial_t m - \Delta m = u\,(d_{11}u + d_{12}v), $$
together with the (homogeneous) Dirichlet boundary condition $m = 0$
on $[0,T] \times \partial\Omega$, and the initial datum $m(0,\cdot) = 0$. By the minimum principle,
it is clear that $m\ge 0$. 
 \medskip

 Defining \(\mu = d_1 +\sigma v\) and $\nu := \frac{\mu\,u +m}{u+m} $, we observe that
 $\min(1,d_1) \le \nu \le \max(1, d_1+ \sigma\,\|v\|_{\infty})$, and $u+m$ satisfies the equation
$$ \partial_t(u+m) - \Delta ( \nu \, (u+m) ) = r_u\,u, $$
together with the (homogeneous) Dirichlet boundary condition $\mu\,u + m = 0$ on
$[0,T] \times \partial\Omega$, and the initial condition $(u+m)(0,x) \equiv u_{\init}(x)$. As a
consequence, the improved duality Lemma~\ref{thm:improved-duality} proved in Appendix~\ref{sec:appendix} ensures that
\begin{equation} \label{ulm}
u, m\in \fsL^{2+\delta}([0,T]\times\Omega),
\end{equation}
for some \(\delta>0\).

We now consider the quantity $w := \int_0^t (\mu u + m)$, and observe that $w\ge 0$ and
$\partial_t w \ge 0$. Moreover
$\Delta w = \int_0^t \Delta (\mu u + m) = \int_0^t [\partial_t (u + m) - r_u\,u] = u + m - u_{in} - r_u \,\int_0^t u$.

It satisfies therefore the parabolic equation
\begin{equation}
  \label{eq:evolution-w2}
 \nu^{-1}\, \partial_t w - \Delta w =  u_{\init} +  r_u \,\int_0^t u,
\end{equation}
together with homogeneous Dirichlet boundary conditions, and the initial condition $w(0, \cdot) = 0$.
\medskip

Observing that $\int_0^t u$ lies in $\fsL^{\infty}([0,T]; \fsL^{2+\delta}(\Omega))$ for $\delta>0$ small enough, we see that
we can use \cref{thm:thm-dirichlet-bc} with $p =\infty$ and $q= 2 + \zeta$ for $\zeta>0$ small enough, recalling that
$d \le 4$, and deduce from it that $\|w\|_{\Ccal^{0,\alpha}([0,T] \times \overline{\Omega})} \le C$,
for some $\alpha, C>0$.

Then, we use the estimate
$$ 0 \le u \le u+m = \Delta w  +  u_{\init} +    r_u \,\int_0^t u 
\le  \Delta \tilde{w}, $$
where
\begin{equation}\label{stara}
  \tilde{w} := w +  \frac{|x|^2}{2d}\, \|u_{\init}\|_{\infty}
  + r_u \, \Delta^{-1} \int_0^t u,
\end{equation}
and  $\Delta^{-1}$ is defined as the operator going from $L^2(\Omega)$ towards itself, which to a function associates the (unique) solution to the Poisson equation with (homogeneous) Dirichlet boundary conditions.
\medskip

Observing that $ \int_0^t u $ lies in $\fsL^{\infty}([0,T]; \fsL^{2+ \kappa}(\Omega))$ for $\kappa>0$
small enough, we see that $\Delta^{-1} \int_0^t u$ lies in
$\fsL^{\infty}([0,T]; \fsW^{2,2+\delta}(\Omega))$ for some $\delta>0$ and therefore in
$\fsL^{\infty}([0,T]; \Ccal^{0,\alpha}(\overline{\Omega}))$, for some $\alpha >0$, in dimension
$d \le 4$, by Sobolev embedding. The same holds for the function
$(t,x) \mapsto \frac{|x|^2}{2d}$.

We now use a so-called one-sided interpolation proven in \cite{boutondesvillettes2025}, which ensures that (for a given time $t$)
\begin{equation}\label{iare}
  \|u \|_{\fsL^{2\,\frac{3 - \alpha}{2-\alpha}} (\Omega)}^3
  \le   C\, \bigg( \,  \| \tilde{w} \|_{\Ccal^{0,\alpha}(\overline{\Omega})}^{\frac3{3-\alpha}}\,
  \| \nabla u\|_{\fsL^2(\Omega)}^{3\frac{2 -\alpha}{3-\alpha}} +
  \| \tilde{w} \|_{\Ccal^{0,\alpha}(\overline{\Omega})}^{3} \,\bigg) .
\end{equation}
We end up the proof as in the previous subsection. Using estimate \eqref{eqp}
for $p=1$ together with estimate \eqref{iare} ensures
(cf.\ \cite{boutondesvillettes2025}) that $u \in \fsL^{3+\delta}([0,T] \times \Omega)$ for some
$\delta>0$, which is sufficient to get an estimate (for $u$) in all $\fsL^p([0,T] \times \Omega)$
when $p\in [1, \infty))$, and to conclude.
\end{proof}

\appendix
\section{Improved duality lemma with Dirichlet boundary conditions}\label{sec:appendix}

Typically, like in \cite{canizo2014}, duality lemmas are proven with
Neumann boundary conditions.  The adaptation to Dirichlet boundary conditions is
direct and, for completeness, we provide a precise statement with a short proof.

\begin{lemm}\label{thm:improved-duality}
  Consider a bounded, \(\Ccal^2\) domain \(\Omega \subset \R^d\) and \(T>0\).  On
  \((0,T] \times \Omega\), consider also a diffusion coefficient \(\mu\) which is bounded above and below, that is,
\(\mu_- \le \mu \le \mu_+\)
for constants
  \(\mu_-,\mu_+ \in (0,\infty)\).  Then, there exists  \(\delta>0\) such
  that a solution \(u\) to
  \begin{equation}
    \left\{
      \begin{aligned}
        &\partial_t u - \Delta(\mu u) = f & & \text{in } (0,T] \times \Omega, \\
        &u = 0 & & \text{on } (0,T] \times \partial\Omega, \\
        &u = u^{\init} & & \text{on } \{0 \} \times \Omega,
      \end{aligned}
    \right.
  \end{equation}
  for
  %compatible \textcolor{blue}{Je ne crois pas qu'on ait besoin de compatibilit\'e pour des données dans $L^p$. L.}
  initial data \(u^{\init} \in \fsL^{2+\delta}(\Omega)\), satisfies
  \begin{equation}
    \lVert u \rVert_{\fsL^{2+\delta}([0,T]\times\Omega)}
    \lesssim
    \lVert u^{\init} \rVert_{\fsL^{2+\delta}(\Omega)}
    + \lVert f \rVert_{\fsL^p((0,T];\fsL^q(\Omega))}
  \end{equation}
  for \(p,q \in [1,\infty]\) satisfying
  \(\frac{2+d}{2+\delta} \ge \frac{2}{p} + \frac{d}{q} - 2\).
\end{lemm}
\begin{proof}
  By rescaling time, we can assume that \(\lambda \le \mu \le 2-\lambda\) for \(\lambda > 0\).  We can
  then rewrite the equation as
  \begin{equation*}
    (\partial_t - \Delta) u = \Delta[(\mu-1)u] + f .
  \end{equation*}
  Recalling the fundamental solution \(\Gamma_{\Omega,D}\) of the heat equation with Dirichlet boundary conditions, the unknown $u$  is determined by
  \begin{equation}\label{eq:u-duality}
    u = \Gamma_{\Omega,D} *_{t,x} \Delta[(\mu-1)u]
    + \Gamma_{\Omega,D} *_{t,x} f
    + \Gamma_{\Omega,D}(t,\cdot) *_x u^{\init}.
  \end{equation}

  By maximal regularity, we find a constant \(C_3 > 0\) such that
  \begin{align*}
    \lVert \Gamma_{\Omega,D} *_{t,x} \Delta[(\mu-1)u] \rVert_{\fsL^3((0,T]\times \Omega)}
    &\le C_3 \,
    \lVert (\mu-1)u \rVert_{\fsL^3((0,T]\times \Omega)}\\
    &\le C_3 \, (1-\lambda)
    \lVert u \rVert_{\fsL^3((0,T]\times \Omega)}.
  \end{align*}
  Moreover, the Laplace operator with Dirichlet boundary conditions has an
  associated orthonormal basis of \(\fsL^2(\Omega)\) consisting of eigenvectors.
  Using this basis, we can verify that
  \begin{align*}
    \lVert \Gamma_{\Omega,D} *_{t,x} \Delta[(\mu-1)u] \rVert_{\fsL^2((0,T]\times \Omega)}
    &\le
    \lVert (\mu-1)u \rVert_{\fsL^2((0,T]\times \Omega)}\\
    &\le (1-\lambda) \,
    \lVert u \rVert_{\fsL^2((0,T]\times \Omega)}.
  \end{align*}
  Hence, by interpolation, we can find some \(\delta>0\) such that
  \begin{equation*}
    \lVert \Gamma_{\Omega,D} *_{t,x} \Delta[(\mu-1)u] \rVert_{\fsL^{2+\delta}((0,T]\times \Omega)}
    \le (1-\lambda/2) \,
    \lVert u \rVert_{\fsL^{2+\delta}((0,T]\times \Omega)}.
  \end{equation*}
  As the factor appearing in the right-hand side of this inequality is less than
  \(1\), we can thus absorb the corresponding term in
  \eqref{eq:u-duality}. 
%\textcolor{purple}
{Classical properties of the heat
    kernel, see, e.g., \cite[Lemma~11]{boutondesvillettes2025}, allow to bound
    the two last terms in the right-hand side of \eqref{eq:u-duality} and to
    obtain the claimed result.}
\end{proof}

\section{Acknowledgements}

The results contained in the present paper have been partially
presented in WASCOM 2025.

On behalf of all authors, the corresponding author states that there
is no conflict of interest.

\printbibliography
\end{document}